\documentclass[12pt]{amsart}
\usepackage{amstext}
\usepackage{mathrsfs}
\usepackage{amsmath} 
\usepackage{hyperref}
\usepackage{graphicx}
\usepackage{amssymb}
\usepackage{yfonts}
\usepackage{xcolor} 
\usepackage{tikz}
\usetikzlibrary{arrows,decorations.pathmorphing,backgrounds,positioning,fit,matrix}
\usetikzlibrary{automata,topaths}
\usetikzlibrary{positioning,calc}

\newtheorem{definition}{Def\text{}inition}[section]
\newtheorem{theorem}[definition]{Theorem}

\newtheorem{lemma}[definition]{Lemma}
\newtheorem{proposition}[definition]{Proposition}
\newtheorem{corollary}[definition]{Corollary}

\newtheorem{question}[definition]{Question}

\begin{document}
	
	\title[Cellular-Lindel\"of spaces]{Consistency and independence phenomena involving cellular-Lindel\"of spaces}
	\author{Rodrigo Hern\'andez-Guti\'errez}\address{Departamento de Matem\'aticas, Universidad Aut\'onoma Metropolitana Campus Iztapalapa, Av. San Rafael Atlixco 186, Leyes de Reforma 1a Secci\'on, Iztapalapa, 09310, Mexico City, Mexico}
	\email{rod@xanum.uam.mx}
	
	\author{Santi Spadaro}\address{Dipartimento di Ingegneria,
		Universit\` a di Palermo, Viale delle Scienze, Ed. 8, 90128, Palermo, Italy}
	\email{santidomenico.spadaro@unipa.it, santidspadaro@gmail.com}

	
	\begin{abstract}
The cellular-Lindel\"of property is a common generalization of the Lindel\"of property and the countable chain condition that was introduced by Bella and Spadaro in 2018. We solve two questions of Alas, Guti\'errez-Dom\'inguez and Wilson by constructing consistent examples of a normal almost cellular-Lindelöf space which is neither cellular-Lindelöf nor weakly Lindel\"of and a Tychonoff cellular-Lindel\"of space of Lindel\"of degree $\omega_1$ and uncountable weak Lindel\"of degree for closed sets. We also construct a ZFC example of a space for which both the almost cellular-Lindel\"of property and normality are undetermined in ZFC. 
     \end{abstract}
	
	\maketitle
	
	\section{Introduction}
	
	The Lindel\"of property and the countable chain condition are two of the most important \emph{smallness} properties in topology. They are a prominent example of how the topological structure of a space can impact its cardinality, as shown by Hajnal and Juh\'asz's inequality stating that every first-countable Hausdorff space with the countable chain condition has cardinality at most continuum and Arhangel'skii's Theorem stating that every Lindel\"of first-countable Hausdorff space has cardinality bounded by the continuum. That is why, at least since Bell, Ginsburg and Woods's 1979 paper \cite{BGW}, there has been considerable interest in studying common generalizations of those two properties (see also \cite{BS3}). A first attempt at that was done by introducing the weak Lindel\"of property. Recall that the  \emph{weak Lindel\"of degree} of a space $X$, denoted by $wL(X)$, is the minimum cardinal $\kappa$ such that for every open cover $\mathcal{U}$ of $X$ there is a $\leq \kappa$ sized subcollection $\mathcal{V}$ of $\mathcal{U}$ such that $X \subset \overline{\bigcup \mathcal{V}}$. A space having countable weak Lindel\"of degree is said to be \emph{weakly Lindel\"of}. Both Lindel\"of spaces and spaces with the countable chain condition are weakly Lindel\"of. However, unlike the Lindel\"of property, the weak Lindel\"of property is not inherited by closed sets. This justifies the introduction of the following cardinal invariant.
	
	\begin{definition}
	\cite{A} Given a space $X$, the weak Lindel\"of degree for closed sets, which we denote by $wL_c(X)$, is the minimum cardinal $\kappa$ such that, for every closed subset $F$ of $X$ and for every open cover $\mathcal{U}$ of $F$, there is a $\leq \kappa$ sized subcollection $\mathcal{V}$ of $\mathcal{U}$ such that $F \subset \overline{\bigcup \mathcal{V}}$.
	\end{definition}
	
	Obviously $wL(X) \leq wL_c(X)$ and it is easy to prove that $wL_c(X)=wL(X)$ for every normal space $X$.
	
	Bell Ginsburg and Woods \cite{BGW} proved that $|X| \leq 2^{wL(X) \cdot \chi(X)}$ for every normal space $X$ and Alas \cite{A} proved that $|X| \leq 2^{wL_c(X) \cdot \chi(X)}$ for every Urysohn space $X$. However, the following questions are still open:
	
	\begin{question}
	(Bell, Ginsburg and Woods, \cite{BGW}) Let $X$ be a regular space. Is $|X| \leq 2^{wL(X) \cdot \chi(X)}$?
	\end{question}
	
	\begin{question}
	(Arhangel'skii, \cite{Ar}) Let $X$ be a Hausdorff space. Is $|X| \leq 2^{wL_c(X) \cdot \chi(X)}$?
	\end{question}

	A common refinement of the Lindel\"of property and the countable chain condition that has recently received a lot of attention is the \emph{cellular-Lindel\"of property}. Recall that a \emph{cellular family} in a topological space $X$ is simply a family of pairwise disjoint non-empty open subsets of $X$. A space $X$ is said to be (almost) cellular-Lindel\"of if for every cellular family $\mathcal{U}$ in $X$ there is a Lindel\"of subspace $L$ of $X$ such that $U \cap L \neq \emptyset$ for every $U \in \mathcal{U}$ (respectively, if $|\{U \in \mathcal{U}: U \cap L \neq \emptyset \}|=|\mathcal{U}|$).
	
	Cellular-Lindel\"of spaces were introduced by Bella and Spadaro in \cite{BS1}, where the question of whether a cellular-Lindel\"of first-countable space has cardinality at most continuum was first posed. Several researchers have attacked this question offering partial answers to it (see \cite{B}, \cite{BS2}, \cite{JSS}, \cite{S}, \cite{TW}, \cite{XS1} and \cite{XS2} for example). Besides, the same authors have embarked in a systematic study of the cellular-Lindel\"of and related properties, showing, among other things, that it has an interesting behavior with the product operation (see \cite{AJPW} and \cite{DS}). In particular, Dow and Stephenson \cite{DS} constructed an example of a cellular-Lindel\"of space whose product with the one point-compactification of a discrete space is not cellular-Lindel\"of.
	
	Almost cellular-Lindel\"of spaces were introduced by Alas, Guti\'errez-Dom\'inguez and Wilson in \cite{AGDW}. They are an interesting subclass of both the class of feebly Lindel\"of spaces and that of cellular-Lindel\"of spaces. The authors of \cite{AGDW} showed that the Mr\'owka-Isbell $\Psi$-space over a MAD family on $\omega_1$ is an example of an almost cellular-Lindel\"of space which is neither cellular-Lindel\"of nor weakly Lindel\"of, but left open the existence of a normal example with the same features. Assuming the existence of a tower of length $\omega_2$ of uncountable subsets of $\omega_1$, we will construct such an example, thus giving a consistent answer to their question.
	
     On the positive side, the authors of \cite{AGDW} proved that almost-cellular-Lindelöf space of Lindel\"of number at most $\omega_1$ are weakly Lindel\"of. That begs the question whether such spaces have also countable weak Lindel\"of number for closed sets. By exploiting scales, we will construct a consistent example of an even cellular-Lindel\"of space of Lindel\"of number $\omega_1$ whose weak Lindel\"of number for closed sets is uncountable.

	We also construct a ZFC example of a space for which both normality and the almost cellular-Lindelöf property are undecided in ZFC. That will be a byproduct of a characterization of when removing a point from a Lindel\"of $P$-space results in a cellular-Lindel\"of space. Even just for normality this appears to be new. 
	
	All spaces are assumed to be Hausdorff. Our notation regarding cardinal functions follows \cite{J}. In particular, given a topological space $X$, $\chi(X)$, $\pi \chi(X)$, $L(X)$ and $c(X)$ denote the character, $\pi$-character, Lindel\"of degree and cellularity of $X$ respectively.

		\section{The Main Results}

	Given a space $X$ and subsets $S, Q \subset X$, the set $S$ is said to be \emph{concentrated on $Q$} if for every open set $U$ with $Q \subset U$ the set $
	S\setminus U$ is countable. It was proved by Rothberger in 1939 that under $\mathfrak{b}=\omega_1$ it is possible to construct an uncountable set of irrational numbers that is concentrated on the rationals. The following discussion can be found in \cite{Ts} in more detail and generality.
	
	Let $\mathcal{P}(\omega)$ be the power set of $\omega$. The Cantor set topology on $\mathcal{P} (\omega)$ is the topology whose base is the collection $\mathcal{B}$ of all sets of the form
	$$[x; n]=\{y \in \mathcal{P}(\omega): y \cap n= x \cap n \}$$
	for $x \in \mathcal{P}(\omega)$ and $n < \omega$.
	
	We will identify each infinite subset of $\omega$ with its natural enumerating function, that is, given $x \in [\omega]^\omega$, define $x \in \omega^\omega$ as follows: $x(0)=\min{x}$ and $x(n+1)=\min{(x \setminus x(n))}$, for every $n<\omega$. Given $x, y \in [\omega]^\omega$ we write $x \leq^* y$ if the set $\{n<\omega: y(n) < x(n)\}$ is finite. The cardinal $\mathfrak{b}$ is the minimal size of a $\leq^*$-unbounded family in $[\omega]^\omega$.
	
	A set $\{x_\alpha: \alpha < \mathfrak{b} \} \subset [\omega]^\omega$ is called a $\mathfrak{b}$-scale if it is $\leq^*$-unbounded and $x_\alpha \leq^* x_\beta$, whenever $\alpha < \beta < \mathfrak{b}$. It is well-known that there are $\mathfrak{b}$-scales in ZFC and if $\mathfrak{b}=\omega_1$ then any $\mathfrak{b}$-scale is concentrated on $[\omega]^{<\omega}$.
	
	The following theorem answers Question 2.4 in \cite{AGDW}.
	
	\begin{theorem}
	$\mathfrak{b}=\omega_1$ implies there is a cellular-Lindel\"of Tychonoff space $X$ such that $L(X) \leq \omega_1$ but $wL_c(X)=\omega_1$.
	\end{theorem}
	
	\begin{proof}
	As before, denote by $\mathcal{B}$ the standard base for the Cantor set topology on $\mathcal{P}(\omega)$. Let 
	$$X=[\omega]^{<\omega} \cup (\omega_1 \times [\omega]^\omega).$$
	We define a topology on $X$ as follows. A basic open neighbourhood for a point $(\alpha, x) \in \omega_1 \times [\omega]^\omega$ has the form $\{\alpha\} \times (B \cap [\omega]^\omega)$, where $B \in \mathcal{B}$. A basic open neighbourhood of a point $x \in [\omega]^{<\omega}$ is of the form
	$$(B \cap [\omega]^{<\omega}) \cup ((\omega_1 \setminus F) \times (B \cap [\omega]^\omega)),$$
	where $B \in \mathcal{B}$, $x \in B$ and $F \in [\omega_1]^{<\omega}$. Notice that $X$ is a Hausdorff space with a base of clopen sets, so it is zero-dimensional and, in particular, completely regular.
	
	We make two observations about the topology of $X$:
	
	\begin{enumerate}
	\item $\{\alpha \} \times [\omega]^\omega$ is a clopen subset of $X$ which is homeomorphic to the irrational numbers, for each $\alpha < \omega_1$, and
	\item $[\omega]^{<\omega}$ is a closed and nowhere dense subset of $X$ that is countable (in fact, homeomorphic to $\mathbb{Q}$).
	\end{enumerate}
	
	The set $\omega_1 \times \{\omega\}$ is closed discrete and of size $\omega_1$, so $wL_c(X) \geq \omega_1$. To see that $L(X) \leq \omega_1$, notice that $\{\alpha \} \times [\omega]^\omega$ is homeomorphic to the irrationals for each $\alpha < \omega_1$, so $X$ is the union of $\omega_1$ many Lindel\"of subspaces.
	
	Let us now prove that $X$ is cellular-Lindel\"of. Let $\mathcal{U}$ be a cellular family in $X$; we need to find a Lindel\"of subspace $L$ of $X$ which intersects each member of $\mathcal{U}$. Since $\omega_1 \times [\omega]^\omega$ is an open dense subset of $X$ with cellularity $\omega_1$, we may assume, without loss of generality, that $|\mathcal{U}|=\omega_1$ and that each member of $\mathcal{U}$ is a basic clopen subset of some $\{\beta\} \times [\omega]^\omega$.  Let $\{U_\alpha: \alpha < \omega_1 \}$ be an enumeration of $\mathcal{U}$. Then, for every $\alpha < \omega_1$, there are $\beta_\alpha < \omega_1$, $n_\alpha < \omega$ and $s_\alpha \in [\omega]^\omega$ such that
	$$U_\alpha=\{\beta_\alpha\} \times ([s_\alpha; n_\alpha] \cap [\omega]^\omega).$$

	Note that, since $[\omega]^\omega$ is a ccc space, the map $\alpha \longmapsto \beta_\alpha$ is countable-to-one.
	In order to define $L$, we construct a set of irrational numbers concentrated on the rationals. 
	
	\vspace{.2in}
	
	{\bf Claim.} There is a $\mathfrak{b}$-scale $Y=\{y_\alpha: \alpha < \omega_1 \} \subset [\omega]^\omega$ such that $y_\alpha \in [s_\alpha; n_\alpha]$, for every $\alpha < \omega_1$
	
	\begin{proof}[Proof of Claim]
	
	We can use $\mathfrak{b}=\omega_1$ to fix a $\mathfrak{b}$-scale $\{x_\alpha: \alpha < \omega_1 \} \subset [\omega]^\omega$. We recursively construct another $\mathfrak{b}$-scale $Y=\{y_\alpha: \alpha < \omega_1 \} \subset [\omega]^\omega$ with the additional property that $y_\alpha \cap n_\alpha = s_\alpha \cap n_\alpha$, for all $\alpha < \omega_1$ in the following way.
	
	Assume that $\{y_\alpha: \alpha < \beta \}$ has been chosen for some $\beta < \omega_1$. Since the set $S_\beta=\{y_\alpha: \alpha < \beta \} \cup \{x_\beta\}$ has cardinality less than $\mathfrak{b}$ we can find a $\leq^*$-bound $b_\beta \in [\omega]^\omega$ for $S_\beta$. Let $m_\beta=|s_\beta \cap n_\beta|$. We define $y_\beta$ in the following way. First, we choose the first $m_\beta$ elements of $y_\beta$ in such a way that $y_\beta \cap n_\beta=s_\beta  \cap n_\beta$. After this, we let $y_\beta(m_\beta+i)=n_\beta+b_\beta(m_\beta+i)$, for each $i < \omega$. In this way, $y_\beta$ is a well-defined infinite subset of $\omega$, $y_\beta \in [s_\beta; n_\beta]$ and $b_\beta \leq^* y_\beta$. It is therefore clear that $Y$ is a $\mathfrak{b}$-scale. 
	\renewcommand{\qedsymbol}{$\triangle$}
	\end{proof}
	
	Let now $L=\{(\beta_\alpha, y_\alpha): \alpha < \omega_1 \} \cup [\omega]^{<\omega}$ and note that $L$ intersects every member of $\mathcal{U}$. We claim that $L$ is Lindel\"of. Let $\mathcal{V}$ be an open cover of $L$ consisting of basic open sets. For each $x \in [\omega]^{<\omega}$ choose $V_x \in \mathcal{V}$ such that $x \in V_x$. Thus, given $x \in [\omega]^{<\omega}$, there exist $t_x \subset \omega$, $k_x < \omega$ and $F_x \in [\omega_1]^{<\omega}$ such that:
	$$V_x=([t_x; k_x] \cap [\omega]^{<\omega}) \cup ((\omega_1 \setminus F_x) \times ([t_x; k_x] \cap [\omega]^\omega)).$$
	
	Since $[\omega]^{<\omega}$ is countable, there exists $\lambda < \omega_1$ such that:
	$$\bigcup \{F_x: x \in [\omega]^{<\omega} \} \subset \lambda.$$
	Let $W=\bigcup \{[t_x; k_x]: x \in [\omega]^{<\omega} \}$. Since $Y$ is concentrated on $[\omega]^{<\omega}$ there exists $\eta <\omega_1$ such that $\{ y_\alpha: \eta \leq \alpha < \omega_1 \} \subset W$.
	
	Note first that $\mathcal{V}_0:=\{V_x: x \in [\omega]^{<\omega} \}$ is a countable subcollection of $\mathcal{V}$ that covers $[\omega]^{<\omega} \cup \{(\beta_\alpha, y_\alpha): \eta \leq \alpha < \omega_1, \lambda \leq \beta_\alpha \}$. 
	

	Since the map $\alpha \longmapsto \beta_\alpha$ is countable-to-one, the set $\{(\beta_\alpha, y_\alpha): \alpha < \omega_1, \beta_\alpha < \lambda \}$ is countable; similarly the set $\{(\beta_\alpha, y_\alpha): \alpha \leq \eta, \lambda \leq \beta_\alpha\}$ is also countable, so there is a countable subcollection $\mathcal{V}_1$ of $\mathcal{V}$ which covers both sets. It turns out that $\mathcal{V}_0 \cup \mathcal{V}_1$ is a countable subcollection of $\mathcal{V}$ which covers $L$.
\end{proof}

We now turn to the construction of a normal almost-cellular Lindel\"of space which is neither cellular-Lindel\"of nor weakly Lindel\"of. The authors of \cite{AGDW} proved that if $\mathcal{A}$ is a maximal almost disjoint family of uncountable subsets of $\omega_1$ then the Mr\'owka-Isbell $\Psi$-space over $\mathcal{A}$ is a Tychonoff non-normal example with all the required properties. We begin by proving that MADness of $\mathcal{A}$ is actually equivalent to almost cellular-Lindel\"ofness of the associated $\Psi$-space, which shows that different combinatorial tools are needed to get a normal example.

Let $\kappa$ be an infinite cardinal and let $\mathcal{A}$ be a family of $\kappa$-sized subsets of $\kappa$. Recall that $\mathcal{A}$ is said to be an almost disjoint family (AD family, for short) if $|A \cap B| < \kappa$, for every $A, B \in \mathcal{A}$ with $A \neq B$. An AD family is said to be a MAD family if it is a maximal AD family. 

Given an AD family $\mathcal{A}$, the $\Psi$-space over $\mathcal{A}$ is the space $\Psi(\mathcal{A})=\mathcal{A} \cup \kappa$, where points of $\kappa$ are isolated and a basic neighbourhood of a point $A \in \mathcal{A}$ is $\{A\} \cup (A \setminus F)$, where $|F| < \kappa$.

\begin{proposition}
Let $\mathcal{A}$ be an almost disjoint family on $\omega_1$. Then $\Psi(\mathcal{A})$ is almost cellular-Lindel\"of if and only if $\mathcal{A}$ is maximal almost disjoint.
\end{proposition}

\begin{proof}
The reverse implication was proved in \cite{AGDW}. For the direct implication, assume that $\Psi(\mathcal{A})$ is almost cellular-Lindel\"of and let $B$ be an uncountable subset of $\omega_1$. Then $\{\{\alpha\}: \alpha \in B \}$ is a cellular family in $\Psi(\mathcal{A})$ and therefore there is a Lindel\"of subspace $L \subset \Psi(\mathcal{A})$ and an uncountable set $C \subset B$ such that $L \cap \{\alpha \} \neq \emptyset$, for every $\alpha \in C$. 

Note that $\mathcal{A}$ is a closed discrete subset of $\psi(\mathcal{A})$. Therefore $L \cap \mathcal{A}$ is Lindel\"of and hence countable. It turns out that there are a countable family $\{A_n: n < \omega \} \subset \mathcal{A}$ and a countable set $F \subset \omega_1$ such that $L \subset \bigcup \{\{A_n\} \cup A_n: n < \omega \} \cup F$. Now, since $C$ is uncountable, there must be $n_0 < \omega$ and $D \subset C$ such that $\beta \in A_{n_0}$, for every $\beta \in D$. It follows that $A_{n_0} \cap B$ is uncountable, and therefore $\mathcal{A}$ is a maximal almost disjoint family.
\end{proof}

Note that the $\Psi$-space over an almost disjoint family on $\omega$ is always cellular-Lindel\"of, because it's separable, whereas the $\Psi$-space over an almost disjoint family on $\omega_1$ is never cellular-Lindel\"of.

Let $\kappa$ be a regular uncountable cardinal and let $A$ and $B$ be subsets of $\kappa$. We say that $A$ is almost contained in $B$, and we write $A \subseteq^* B$, if the set $A \setminus B$ has cardinality strictly smaller than $\kappa$.

Let $\mathcal{F}$ be a subfamily of $[\kappa]^\kappa$. We say that $\mathcal{F}$ has the strong intersection property if for every subfamily $\mathcal{G}$ of $\mathcal{F}$ of size $< \kappa$, the intersection $\bigcap \mathcal{G}$ has cardinality $\kappa$. Moreover, we say that $P \in [\kappa]^\kappa$ is a \emph{pseudointersection} for $\mathcal{F}$ if $P \subseteq^* F$, for every $F \in \mathcal{F}$. The cardinal $\mathfrak{p}_\kappa$ is defined as the minimal cardinality of a subfamily in $[\kappa]^\kappa$ which has no pseudointersection.

A sequence $\mathcal{T}=\{T_\alpha: \alpha < \lambda \} \subset [\kappa]^\kappa$ with the strong intersection property is called a tower if $T_\beta \subseteq^* T_\alpha$, for every $\alpha < \beta < \lambda$. The \emph{tower number} $\mathfrak{t}_\kappa$ is defined as the minimal cardinality of a tower $\mathcal{T} \subset [\kappa]^\kappa$. 

It is easy to see that these cardinals are well-defined and $\kappa^+ \leq \mathfrak{p}_\kappa \leq \mathfrak{t}_\kappa \leq 2^{\kappa}$ (see \cite{Sch}). So, in particular $\mathfrak{p}_\kappa=\mathfrak{t}_\kappa=\kappa^+$ if the GCH holds at $\kappa$. Malliaris and Shelah proved that $\mathfrak{p}_\omega=\mathfrak{t}_\omega$ in ZFC, but the question of whether $\mathfrak{p}_\kappa=\mathfrak{t}_\kappa$ in ZFC for every regular cardinal $\kappa$ (or even for $\kappa=\omega_1$) is still open. For some progress on that question, as well as a model of ZFC where $\mathfrak{p}_\kappa=\mathfrak{t}_\kappa=\kappa^+< 2^\kappa$ see \cite{FMSS}.



The following theorem consistently answers the first half of Question 2.8 from \cite{AGDW}.

\begin{theorem}
$\mathfrak{t}_{\omega_1}=\omega_2$ implies there is a normal almost cellular-Lindel\"of space $X$ which is neither cellular-Lindel\"of nor weakly Lindel\"of.
\end{theorem}

\begin{proof}
Fix a tower $\mathcal{T}'=\{T_\alpha: \alpha < \omega_2 \}$ of uncountable subsets of $\omega_1$. Let $S=\{\alpha < \omega_2: cf(\alpha) \neq \omega \}$ and let $\mathcal{T}=\{T_\alpha: \alpha \in S \}$. Let $X=\mathcal{T} \cup \omega_1$. Declare every point of $\omega_1$ to be isolated and declare a basic neighbourhood of a point $T_\alpha \in \mathcal{T}$ to be 
$$U(\beta, \alpha, C):=\{T_\gamma: \gamma \in (\beta, \alpha] \cap S\} \cup T_\beta \setminus (T_\alpha \cup C)$$
where $\beta < \alpha$ and $C$ is a countable set. 

\vspace{.2in}

\noindent {\bf Claim 1} The space $X$ is a $P$-space.

\begin{proof}[Proof of Claim 1] It suffices to prove that every point of $X$ is a $P$-point. This is obviously true for every point of $\omega_1$, as well as for every $T_\alpha$ such that $\alpha$ is a successor ordinal. Now fix a limit ordinal $\alpha \in S$, ordinals $\beta_n < \alpha$ and a countable subset $C_n$ of $X$, for every $n< \omega$.  We will prove that $\bigcap \{U(\beta_n, \alpha, C_n): n < \omega \}$ is an open neighbourhood of $T_\alpha$. Indeed, let $\beta=\sup \{\beta_n: n < \omega \}<\alpha$ and let $F_n= T_\beta \setminus T_{\beta_n}$, which is a countable set. Finally, let $C=\bigcup \{C_n \cup F_n: n < \omega \}$. Then
$$T_\alpha \in U(\beta, \alpha, C) \subset \bigcap \{U(\beta_n, \alpha, C_n): n < \omega \}$$
which proves that $T_\alpha$ is a $P$-point.
\renewcommand{\qedsymbol}{$\triangle$}
\end{proof}

\noindent {\bf Claim 2.} The space $X$ has countable extent.

\begin{proof}[Proof of Claim 2] First of all note that the subspace $\mathcal{T}$ is closed in $X$ and has countable extent. Indeed, $\mathcal{T}$ is homeomorphic to $S$ with the order topology. Now let $A$ be a subspace of $S$ having cardinality $\omega_1$; there must be an ordinal $\gamma < \omega_2$ such that $A \subset \gamma$. But then $A$ is a subset of $S \cap (\gamma+1)$, which is a Lindel\"of space. It turns out that $A$ must have an accumulation point.

Now suppose by contradiction that $X$ contains an uncountable closed discrete subset $D$. We can assume $|D|=\omega_1$ and, by the above argument, we can also assume that $D \subset \omega_1$. Since $\mathcal{T}$ has no uncountable pseudointersection, we can consider the minimum ordinal $\alpha < \kappa$ such that $D \setminus T_\alpha$ is uncountable. It is easy to see that $T_\alpha$ is an accumulation point of $D$ (as a matter of fact, $D$ is a transfinite sequence converging to $T_\alpha$).

\renewcommand{\qedsymbol}{$\triangle$}
\end{proof}

\noindent {\bf Claim 3.} The space $X$ is not cellular-Lindel\"of.
	
	\begin{proof}[Proof of Claim 3]
	
Let $\mathcal{U}=\{\{\alpha\}: \alpha < \omega_1 \}$ which is a cellular family in $X$. Suppose by contradiction that there exists a Lindel\"of subspace $L$ of $X$ such that $L \cap U \neq \emptyset$, for every $U \in \mathcal{U}$. Then $\omega_1 \subset L$, so, in particular $L$ is dense in $X$. But since $L$ is a Lindel\"of subspace of a $P$-space then $L$ must be closed in $X$. It follows that $L$ must be equal to $X$, but that's a contradiction, since $X$ is easily seen not to be Lindel\"of: it suffices to note that $\mathcal{T}$ is closed in $X$ and homeomorphic to $S$ with the order topology.
	\renewcommand{\qedsymbol}{$\triangle$}
	\end{proof}
	
	\noindent {\bf Claim 4.} The space $X$ is almost cellular-Lindel\"of.

         \begin{proof}[Proof of Claim 4] Let $\{U_\alpha: \alpha < \omega_1 \}$ be a cellular family. Since $X$ is a $P$-space we can assume each $U_\alpha$ is clopen. By countable extent there must be $\tau < \omega_2$ such that $T_\tau \in \overline{\bigcup \{U_\alpha: \alpha < \omega_1\}} \setminus \bigcup \{U_\alpha: \alpha < \omega_1 \}$, or otherwise $\{U_\alpha: \alpha < \omega_1 \}$ would be a discrete family. Since $\chi(T_\tau, X)=\omega_1$, there is an $\omega_1$-sequence $S \subset \bigcup \{U_\alpha: \alpha < \omega_1\}$ which converges to $T_\tau$. Moreover, since $T_\tau \notin U_\alpha= \overline{U_\alpha}$ for every $\alpha < \omega_1$, the intersection $S \cap U_\alpha$ is at most countable, for every $\alpha < \omega_1$. Therefore $L=S \cup \{T_\tau\}$ is a Lindel\"of subspace of $X$ which intersects $\omega_1$ many $U_\alpha$'s. \renewcommand{\qedsymbol}{$\triangle$}

\end{proof}

\noindent {\bf Claim 5.} The space $X$ is not weakly Lindel\"of.

\begin{proof}[Proof of Claim 5]

For every $\alpha \in S$, let $U_\alpha$ be the following open subset of $X$:
$$U_\alpha=U(0, \alpha, \emptyset)=\{T_\gamma: \gamma \in [0, \alpha] \cap S \} \cup T_0 \setminus T_\alpha.$$
Then $\{U_\alpha: \alpha < \omega_2 \} \cup \{\{\beta\}: \beta < \omega_1\}$ is an open cover of $X$. 

Suppose by contradiction that there are countably many ordinals $\{\alpha_n: n < \omega \} \subset S$ and countably many ordinals $\{\beta_n: n < \omega \} \subset \omega_1$ such that 

$$X \subset \overline{\bigcup \{U_{\alpha_n}: n < \omega \} \cup \bigcup \{\{\beta_n\}: n < \omega\}}.$$

Let $\gamma \in S$ be an ordinal such that $\sup \{\alpha_n: n < \omega \} <\gamma$ and let $C=\bigcup \{T_\gamma \setminus T_{\alpha_n}: n < \omega \} \cup \{\beta_n: n < \omega \}$. Then $U(\gamma, \gamma+1, C)$ is an open subset of $X$ which is disjoint from $\bigcup \{U_{\alpha_n}: n < \omega \} \cup \bigcup \{\{\beta_n\}: n < \omega\}$, but that is a contradiction.\renewcommand{\qedsymbol}{$\triangle$}
\end{proof}

\end{proof}

We would like to finish by constructing a ZFC example of a space for which having the cellular-Lindel\"of property is independent of ZFC. Our example is a byproduct of a positive result (Corollary $\ref{corpoint}$) which has independent interest. We finally prove that also the normality of our example is independent of ZFC. It appears that the existence of such an example is also new.

Recall that a space $X$ is said to be \emph{strongly cellular-Lindel\"of} (see \cite{B}) if for every cellular family $\mathcal{U}$ in $X$ there is a \emph{closed} Lindel\"of subspace $L$ of $X$ such that $L \cap U \neq \emptyset$, for every $U \in \mathcal{U}$. Note that every cellular-compact space is strongly cellular-Lindel\"of. Moreover, cellular-Lindel\"ofness and strong cellular-Lindel\"ofness are equivalent for $P$-spaces.

\begin{lemma} \label{plemma}
Let $X$ be a $P$-space. Then $X$ is cellular-Lindel\"of if and only if $X$ is strongly cellular-Lindel\"of.
\end{lemma}

\begin{proof}
It suffices to note that every Lindel\"of subspace of a $P$-space is closed.
\end{proof}

\begin{lemma} 
Let $X$ be a strongly cellular-Lindel\"of space. Then every regular closed subset of $X$ is strongly cellular-Lindel\"of.
\end{lemma}

\begin{proof}
Let $F$ be a regular closed subset of $X$ (that is, $F=\overline{Int(F)}$). Let $\mathcal{G}$ be a cellular family in $F$. Then for every $G \in \mathcal{G}$ there is an open subset $U_G$ of $X$ such that $G=U_G \cap F$. It turns out that $\mathcal{U}=\{U_G \cap Int(F): G \in \mathcal{G} \}$ is a cellular family in $X$ and therefore there is a closed Lindel\"of subspace $L$ of $X$ such that $L \cap U \neq \emptyset$, for every $U \in \mathcal{U}$. Then $L \cap F$ is a closed Lindel\"of subspace of $F$ such that $(L \cap F) \cap G$ is non-empty, for every $G \in \mathcal{G}$ and therefore $F$ is strongly cellular-Lindel\"of.
\end{proof}

\begin{theorem} \label{propdisj}
Let $X$ be a strongly cellular-Lindel\"of regular space. If $p \in X$ does not have a disjoint local $\pi$-base then $X \setminus \{p\}$ is strongly cellular-Lindel\"of.
\end{theorem}

\begin{proof}
Assume that $p$ does not have a disjoint local $\pi$-base and let $\mathcal{U}$ be a cellular family in $X \setminus \{p\}$. Then $\mathcal{U}$ is not a local $\pi$-base at $p$ and hence there is a neighbourhood $V$ of $p$ such that $U \setminus V \neq \emptyset$, for every $U \in \mathcal{U}$. Since $X$ is a regular space we can actually assume that $U \setminus \overline{V} \neq \emptyset$, for every $U \in \mathcal{U}$ and therefore $\{U \setminus \overline{V}: U \in \mathcal{U} \}$ is a cellular family in the regular closed subspace $F=\overline{X \setminus \overline{V}}$ of $X$. By the lemma above we can find a closed Lindel\"of subspace $L \subset F$ such that $L \cap (U \setminus \overline{V}) \neq \emptyset$, for every $U \in \mathcal{U}$. But then $L$ is a closed Lindel\"of subspace of $X \setminus \{p\}$ which meets every member of $\mathcal{U}$ and that proves that $X \setminus \{p\}$ is strongly cellular-Lindel\"of.
\end{proof}

\begin{theorem} \label{propcell}
Let $X$ be a regular $P$-space and $p \in X$ be a non-isolated point. If $X \setminus \{p\}$ is cellular-Lindel\"of then $p$ does not have a disjoint local $\pi$-base.
\end{theorem}

\begin{proof}
Assume that $X \setminus \{p\}$ is cellular-Lindel\"of and suppose by contradiction that $p$ has a disjoint local $\pi$-base $\mathcal{B}$. By thinning out the elements of $\mathcal{B}$ we can assume that $p \notin \overline{B}$ for every $B \in \mathcal{B}$ and therefore $\mathcal{B}$ is a cellular family in $X \setminus \{p\}$. Hence there exists a Lindel\"of subspace $L$ of $X \setminus \{p\}$ such that $L \cap B \neq \emptyset$, for every $B \in \mathcal{B}$. From the definition of a local $\pi$-base, $p \in \overline{L}$. Since a Lindel\"of subspace of a $P$-space is closed, then $L=\overline{L}$. Thus, $p\in L\subset X\setminus\{p\}$. But that is a contradiction and hence we are done.
\end{proof}

The assumption that $X$ is a $P$-space cannot be removed from the proposition above. Indeed, if $X$ is any compact first-countable space without isolated points then, for every point $p \in X$, the space $X \setminus \{p\}$ is $\sigma$-compact and hence strongly cellular-Lindel\"of, but $p$ has a disjoint local $\pi$-base.

\begin{corollary} \label{corpoint}
Let $X$ be a cellular-Lindel\"of $P$-space and $p \in X$ be a point. Then $X \setminus \{p\}$ is cellular-Lindel\"of if and only if $p$ does not have a disjoint local $\pi$-base.
\end{corollary}

\begin{proof}
Combine Theorem $\ref{propdisj}$ and $\ref{propcell}$ with Lemma $\ref{plemma}$.
\end{proof}


Given an infinite cardinal $\kappa$ recall that the $\sigma$-product of $2^\kappa$ is defined as $\sigma(2^\kappa)=\{x \in 2^\kappa: |x^{-1}(1)| < \aleph_0 \}$. Given a space $X$, recall that the $G_\delta$ topology $X_\delta$ is the topology on $X$ generated by the $G_\delta$ subsets of $X$ It is well-known that $\sigma(2^\kappa)_\delta$, is a Lindel\"of space. The easiest way to see this is to note that $\sigma(2^\kappa)$ is the union of countably many compact scattered spaces and use a result of Arhangel'skii stating that the $G_\delta$ topology on a compact scattered space is Lindel\"of (see \cite{ArhCp}).

\begin{corollary} \label{notCH}
Let $\kappa \geq \omega_2$ and $X=\sigma(2^\kappa)_\delta$. Then $X \setminus \{p\}$ is cellular-Lindel\"of, for every $p \in X$.
\end{corollary}

\begin{proof}
Note first that $\pi \chi(p, X)=cof([\kappa]^\omega, \subset) \geq \kappa \geq \omega_2$. To see why the equality is true, then let $\mathcal{U}$ be a local $\pi$-base at $p$. Without loss of generality we can assume that $\mathcal{U}$ is made up of basic open sets and enumerate $\mathcal{U}$ as $\{[\sigma_\alpha]: \alpha < \lambda \}$, where $\sigma_\alpha \in Fn(\kappa, 2, \omega_1)$, for every $\alpha < \lambda$ and given $\sigma \in Fn(\kappa, 2, \omega_1)$, $[\sigma]=\{f \in X: \sigma \subseteq f \}$. Given $A \in [\kappa]^\omega$ we see that  $[\sigma_\alpha] \subseteq [p \upharpoonright A]$ if and only if $p \upharpoonright A \subseteq \sigma$, which is equivalent to $A \subset dom(\sigma)$ and $p(\beta)=\sigma(\beta)$ for every $\beta \in A$. It follows that the family of domains of elements of $\mathcal{U}$ is cofinal in $([\kappa]^\omega, \subseteq)$. Viceversa, if $\mathcal{A}$ is a cofinal family in $([\kappa]^\omega, \subseteq)$ then $\{[p \upharpoonright A]: A \in \mathcal{A} \}$ is even a local base at $p$.

 On the other hand $c(X)=\omega_1$, because $X$ is an uncountable Lindel\"of $P$-group (see \cite{U}). It follows that $p$ cannot have a disjoint local $\pi$-base and hence by Corollary $\ref{corpoint}$ the space $X \setminus \{p\}$ is cellular-Lindel\"of.
\end{proof}

\begin{corollary} \label{CH}
The space $\sigma(2^{\omega_1})_\delta \setminus \{p\}$ is not cellular-Lindel\"of, for every $p \in \sigma(2^{\omega_1})$.
\end{corollary}

\begin{proof}
Being a $P$-space of character $\omega_1$, the space $\sigma(2^{\omega_1})_\delta$ has a disjoint local $\pi$-base at every point (see \cite{BS4}) and therefore the statement follows from Corollary $\ref{corpoint}$.
\end{proof}

As a byproduct we obtain a ZFC example of a space whose cellular-Lindel\"of property is independent of ZFC.

\begin{theorem} \label{independent}
There is a Tychonoff space $X$ such that $X$ is cellular-Lindel\"of if and only if CH fails. 
\end{theorem}

\begin{proof}
It suffices to set $X=\sigma(2^{\mathfrak{c}})_\delta \setminus \{\mathbf{0} \}$ and use Corollaries $\ref{notCH}$ and $\ref{CH}$.
\end{proof}

We finally prove that the construction from Theorem $\ref{independent}$ provides an example of a space whose normality is independent of ZFC.

\begin{theorem}
There is a ZFC example of a space $X$ such that $X$ is normal under CH and $X$ is not normal under not CH.
\end{theorem}

\begin{proof}
Again, let $X=\sigma(2^{\mathfrak{c}})_\delta \setminus \{\mathbf{0} \}$. Under CH, every $P$-space of cardinality continuum is even paracompact (see \cite{I}). Thus, $X$ is normal under CH.

Assume now that $\mathfrak{c}>\aleph_1$, and let $e_\alpha \in 2^{\mathfrak{c}}$ be the function defined as follows:
$$e_\alpha(\beta)=\begin{cases}
0, \textrm{ if } \beta \neq \alpha,\textrm{ and} \\
1, \textrm{ if } \beta=\alpha.
\end{cases}$$
Then $\{e_\alpha: \alpha < \mathfrak{c}\}$ is a closed discrete subset of $X$. Let $\{A, B \}$ be a partition of $\mathfrak{c}$ into two sets such that $|A|=\aleph_1$. We will show that the disjoint closed sets $F_A=\{e_\alpha: \alpha \in A\}$ and $F_B=\{e_\alpha: \alpha \in B \}$ cannot be separated. Indeed, suppose by contradiction that $U_A$ and $U_B$ are disjoint open subsets of $X$ such that $F_A \subset U_A$ and $F_B \subset U_B$.


For every $\alpha \in A$ let $\sigma_\alpha: \mathfrak{c} \to 2$ be a countable partial function such that $e_\alpha \in [\sigma_\alpha] \subset U_A$ and, for every $\beta \in B$, let $\sigma_\beta: \mathfrak{c} \to 2$ be a countable partial function such that $e_\beta \in [\sigma_\beta] \subset U_B$. Without loss of generality we can assume that $\alpha \in dom(\sigma_\alpha)$, for every $\alpha < \mathfrak{c}$.

Use $\mathfrak{c} > \aleph_1$ to find an ordinal $\gamma < \mathfrak{c}$ such that $\gamma \notin \bigcup \{dom(\sigma_\alpha): \alpha \in A \}$. Then $\gamma \in B$ and hence

$$[\sigma_\gamma] \cap [\sigma_\alpha]=\emptyset \textrm{ for every }\alpha \in A.$$
%


Fix $\alpha \in A$ and note that $\gamma \notin dom(\sigma_\alpha)$. Therefore $\sigma_\alpha(\beta)=0$ for every $\beta \in dom(\sigma_\gamma) \cap dom(\sigma_\alpha) \setminus \{\alpha \}$.  It follows that in order for $[\sigma_\gamma]$ to be disjoint from $[\sigma_\alpha]$ we must have $\alpha \in dom(\sigma_\gamma)$.

Since $\alpha$ was an arbitrary element of $A$ we deduce that $A \subset dom(\sigma_\gamma)$, but that contradicts the fact that $\sigma_\gamma$ is a countable partial function.

Therefore $X$ is not normal if CH doesn't hold.

\end{proof}

\section*{Acknowledgements}
Research of the first-named author was supported by a grant of the GNSAGA group of INdAM and the FORDECYT-PRONACES grant 64356/2020 of CONAHCyT. The second-named author was supported by a grant from the GNSAGA group of INdAM and by a grant from the Fondo Finalizzato alla Ricerca di Ateneo (FFR 2024) of the University of Palermo.

	\end{document}